\documentclass[12pt,a4paper]{amsart}
\usepackage{amscd,amsfonts,amsmath,amssymb,latexsym,mathrsfs,pifont,xcolor}
\numberwithin{equation}{section}\theoremstyle{plain}
\newtheorem{theorem}{Theorem}[section]
\newtheorem{proposition}[theorem]{Proposition}

\newtheorem{corollary}[theorem]{Corollary}

\newcommand{\norm}[1]{\left\Vert#1\right\Vert}
\newcommand{\abs}[1]{\left\vert#1\right\vert}
\newcommand{\set}[1]{\left\{#1\right\}}

\title[Riesz transforms for Dunkl transform ]{Riesz transforms for Dunkl transform }
\author[B. Amri, M. Sifi]{B\'{e}chir Amri and Mohamed Sifi}
\address{B\'{e}chir Amri, University of Tunis, Preparatory Institute of Engineer Studies of Tunis,
Department of Mathematics, 1089 Montfleury, Tunis, Tunisia} \email{
bechir.amri@ipeit.rnu.tn}
\address{Mohamed Sifi, University of Tunis El Manar,
Faculty of Sciences of Tunis, Department of Mathematics,
2092 Tunis El Manar, Tunis, Tunisia} \email{mohamed.sifi@fst.run.tn}
\begin{document}



 \thanks{The authors were partially supported by the DGRST project
04/UR/15-02 and the cooperation programs PHC Utique / CMCU 07G 1501 \&
10G 1503.}

\maketitle
\begin{abstract}
In this paper we obtain the $L^p$-boundedness of Riesz transforms for Dunkl transform for all $1<p<\infty$.
\end{abstract}

\section{Introduction}
 On the Euclidean space $\mathbb{R}^N$, $N\geq 1$, the ordinary  Riesz transform  $R_j$, $j=1,...,N$ is defined as the multiplier operator
 \begin{eqnarray}\label{rt}
\widehat{R_j(f)}(\xi)=-i\frac{\xi_j}{\norm{\xi}}\widehat{f}(\xi). \end{eqnarray}
It can  also  be defined by the principal value of the singular integral
$$R_j(f)(x)=d_0\lim_{\varepsilon\rightarrow0}\int_{\norm{x-y}>\varepsilon}\frac{x_j-y_j}{\norm{x-y}}f(y)dy$$
where $d_0=\displaystyle{2^{\frac{N}{2}}\frac{\Gamma(\frac{N+1}{2})}{\sqrt{\pi}}}$. It follows from the general theory of singular integrals
    that Riesz transforms are
bounded on $L^p(\mathbb{R}^N,dx)$ for all $1 < p < Å\infty$. What is done in this paper is to extend this result to the context of Dunkl
 theory
where a similar operator is already defined.
\par Dunkl theory generalizes classical Fourier analysis on $\mathbb{R}^N$. It started twenty years ago
with Dunkl's seminal work \cite{D1} and was further developed by several mathematicians. See
for instance the surveys \cite{J1,R1, R2, xu} and the references cited therein.
The study of the $L^p$-boundedness of Riesz transforms for  Dunkl transform on $\mathbb{R}^N$ goes back to the work
 of S. Thangavelyu and Y. Xu \cite{Tx} where they established  boundedness result only in a very special case of $N=1$.  It has
been noted in \cite{Tx}  that the difficulty  arises  in the application of  the classical $L^p$-
theory of Cald\'{e}ron-Zygmund,
since Riesz transforms are singular integral operators.
 In this paper we describe  how this theory can be adapted in Dunkl setting
 and  gives an $L^p$-result for Riesz transforms for all $1<p<\infty$. More precisely, through  the fundamental  result of
  M. R\"{o}sler \cite{R1} for the Dunkl translation of  radial functions, we reformulate a H\"{o}rmander type condition for  singular
  integral operators.
   The Riesz kernel is given by acting Dunkl operator on  Dunkl translation of radial function.
  \par This paper is organized as follows.
  In Section 2 we present some definitions and fundamental results from  Dunkl's analysis.
     The Section 3  is devoted to proving  $L^p$-boundedness of  Riesz transforms. As applications, we will prove a generalized Riesz
     and
      Sobolev inequalities. Throughout this paper $C$ denotes a constant which can vary from line to line.
      \section{Preliminaries }In this section we collect notations and definitions and recall some basic facts. We refer to
      \cite{ D1, J1, R1, R2,xu}.\par
Let $G\!\subset\!\text{O}(\mathbb{R}^N)$ be a finite reflection
group associated to a reduced root system $R$ and
$k:R\rightarrow[0,+\infty)$ be  a $G$-invariant function (called
multiplicity function). Let $R_+$ be a positive root subsystem. We
shall  assume that $R$ is normalized in the sense that
$\norm{\alpha}^2=\langle\alpha,\alpha\rangle=2$ for all $\alpha\in R
$, where $\langle\;,\;\rangle$ is  the standard Euclidean scalar
product on $\mathbb{R}^N$.
\par The Dunkl operators \,$T_\xi$,  $\xi\in \mathbb{R}^N$ are
the following $k$--de\-for\-ma\-tions of directional derivatives $\partial_\xi$
by difference operators\,:
$$
T_\xi f(x)=\partial_\xi f(x)
+\sum_{\,\alpha\in R_+}\!k(\alpha)\,\langle\alpha,\xi\rangle\,
\frac{f(x)-f(\sigma_\alpha.\,x)}{\langle\alpha,\,x\rangle},\quad x\in\mathbb{R}^N
$$
where   $\sigma_\alpha$
denotes the reflection
with respect to the hyperplane orthogonal to $\alpha$. For the standard basis vectors of $\mathbb{R}^N$, we simply write
$ T_j=T_{e_j} $.
 \par The operators $\partial_\xi$ and $T_\xi$
are intertwined by a Laplace--type operator
\begin{eqnarray*}\label{v}
V_k\hspace{-.25mm}f(x)\,
=\int_{\mathbb{R}^N}\hspace{-1mm}f(y)\,d\mu_x(y),
\end{eqnarray*}
associated to a family of compactly supported probability measures
\,$\{\,\mu_x\,|\,x\!\in\!\mathbb{R}^N\hspace{.25mm}\}$\,.
Specifically, \,$\mu_x$ is supported in the the convex hull $co(G.x)$.
\par For every $\lambda\!\in\!\mathbb{C}^N$\!,
the simultaneous eigenfunction problem,
\begin{equation*}
T_\xi f=\langle\lambda,\xi\rangle\,f,
\qquad \xi\!\in\!\mathbb{R}^N
\end{equation*}
has a unique solution $f(x)\!=\!E_k(\lambda,x)$
such that $E_k(\lambda,0)\!=\!1$,
which is given by
\begin{equation*}\label{EV}
E_k(\lambda,x)\,
=\,V_k(e^{\,\langle\lambda,\,.\,\rangle})(x)\,
=\int_{\mathbb{R}^N}\hspace{-1mm}e^{\,\langle\lambda,y\rangle}\,d\mu_x(y),
\qquad x\!\in\!\mathbb{R}^N.
\end{equation*}
Furthermore \,$\lambda\mapsto E_k(\lambda,x)$ \,extends to
a holomorphic function on \,$\mathbb{C}^N$.
\par Let $m_k$ be the measure on $\mathbb{R}^N$, given by
\begin{eqnarray*}
dm_k(x)=\prod_{\alpha\in R_+}|\langle\alpha,x\rangle|^{2k(\alpha)}dx.
\end{eqnarray*}
For $f\in L^1(m_k)$ (the Lebesgue space with respect to the measure $m_k$)
the Dunkl transform  is defined by
$$\mathcal{F}_k(f)(\xi)=\frac{1}{c_k}
\int_{\mathbb{R}^N}\!f(x)\,E_k(-i\,\xi,x)dm_k(x),\quad c_k\,=\int_{\mathbb{R}^N}\!e^{-\frac{|x|^2}2}\,dm_k(x).
$$
 This new  transform shares many analogous properties of the Fourier transform.
\begin{itemize}
\item[(i)]
The Dunkl transform is a topological automorphism
of the Schwartz space $\mathcal{S}(\mathbb{R}^N)$.
\item[(ii)]
(\textit{Plancherel Theorem\/})
The Dunkl transform extends to
an isometric automorphism of $L^2(m_k)$.
\item[(iii)]
(\textit{Inversion formula\/})
For every
  $f\!\in\!L^1(m_k)$
such that $\mathcal{F}_kf\!\in\!L^1(m_k)$,
we have
$$
f(x)=\mathcal{F}_k^2\!f(-x),\qquad x\!\in\!\mathbb{R}^N.
$$
\item[(iv)]
  For all $\xi\in\mathbb{R}^N$ and $f\in \mathcal{S}(\mathbb{R}^N)$
 \begin{equation}\label{dtj}
  \mathcal{F}_k(T_\xi(f))(x)=<i\xi,x>\mathcal{F}_k(f)(x),\qquad x\!\in\!\mathbb{R}^N.
\end{equation}
\end{itemize}
\par Let $x\in \mathbb{R}^N$, the
Dunkl translation operator $\tau_x$ is defined on  $ L^2(m_k)$  by,
\begin{eqnarray}\label{dutr}
\mathcal{F}_k(\tau_x(f))(y)=
E_k(ix,y)\mathcal{F}_kf(y), \quad y\in\mathbb{R}^N.
\end{eqnarray}
If $f$ is a  continuous radial function in $L^2(m_k)$ with
$f(y)=\widetilde{f}(\norm{y})$, then
\begin{eqnarray}\label{trad}
              \tau_x(f)(y)=\int_{\mathbb{R}^{N}}\widetilde{f}\Big(\;\sqrt{\norm{x}^2+\norm{y}^2+2<y,\eta>}\;\Big)d\mu_x(\eta).
\end{eqnarray}
This formula is first proved by   M. R\"{o}sler \cite{R1}  for  $f\in \mathcal{S}(\mathbb{R}^N)$  and   recently is
extended to continuous functions by F. and H. Dai Wang \cite {DW}.
 \par We collect below some useful facts :
  \begin{itemize}
   \item
  [(i)] For all $x,y\in \mathbb{R}^N$,
  \begin{equation}\label{com}
   \tau_x(f)(y)=\tau_y(f)(x).
\end{equation}
   \item  [(ii)]For all $x,\xi\in \mathbb{R}^N$ and $f\in\mathcal{S}(\mathbb{R}^N)$,
      \begin{equation}\label{dtr}
   T_\xi \tau_x(f)=\tau_x T_\xi(f).
\end{equation}
   \item [(iii)]
   For all $x\in \mathbb{R}^N$ and $f,\;g\in L^2(m_k)$,
   \begin{equation}\label{cov}
\int_{\mathbb{R}^N}\tau_x(f)(-y)g(y)dm_k(y)=\int_{\mathbb{R}^N}f (y) \tau_xg(-y) dm_k(y).
\end{equation}

           \item [(iv)] For all $x\in \mathbb{R}^N$ and  $1\leq p\leq2$, the operator $\tau_x$ can be extended to all radial
           functions $f$ in $L^p(m_k)$  and the following holds
\begin{equation}\label{tp}
  ||\tau_x(f)||_{p,k}\leq ||f||_{p,k}.
  \end{equation}
  \end{itemize}
$\norm{.}_{p,k}$ is the usual norm of $L^p(m_k)$.
  \section{ Riesz transforms  for the Dunkl transform.}
In Dunkl setting the Riesz transforms  (see \cite{Tx}) are the operators $\mathcal{R}_j $,  $j=1...N$ defined on $L^2(m_k)$ by
 $$\mathcal{R}_j(f)(x)=d_k\lim_{\varepsilon\rightarrow0}\int_{|y|>\varepsilon}\tau_x(f)(-y)\frac{y_j}{\norm{y}^{p_k}}dm_k(y),
 \quad x\in\mathbb{R}^N$$
 where
 $$d_k=2^{\frac{p_k-1}{2}}\frac{\Gamma(\frac{p_k}{2})}{\sqrt{\pi}}; \quad p_k=2\gamma_k+N+1\quad \mbox{and}\quad
 \gamma_k=\sum_{\alpha\in R+}k(\alpha).$$
 It has been proved by  S. Thangavelyu and Y. Xu \cite{Tx}, that
$\mathcal{R}_j$  is a multiplier operator given by
\begin{eqnarray}\label{rm}
\mathcal{F}_k(\mathcal{R}_j(f))(\xi)= -i\frac{\xi_j}{\norm{\xi}}\mathcal{F}_k(f)(\xi), \quad f\in\mathcal{S}(\mathbb{R}^N),
\;\xi\in\mathbb{R}^N,
\end{eqnarray}
The authors state that if
  $N=1$ and $2\gamma_k\in\mathbb{N}$  the operator $\mathcal{R}_j$  is bounded   on $L^p(m_k)$, $1<p<\infty$. In  \cite{A}
 this result is improved   by removing $2\gamma_k\in\mathbb{N}$, where Riesz transform is called Hilbert transform. If $\gamma_k=0$ ($k=0$), this operator coincides with the usual Riesz transform $R_j$ given by (\ref{rt}).
Our interest is to prove the boundedness of this operator  for
  $N\geq2$ and $\gamma_k\geq 0$. To do this, we invoke the theory of singular integrals. Our basic is the following,
\begin{theorem}\label{l1} Let $\mathcal{K}$ be a  measurable function
 on $\mathbb{R}^N\times\mathbb{R}^N\setminus \set{(x,g.x);\; x\in\mathbb{R}^N}$ and  $S$ be a bounded
operator
 from $L^2(m_k)$ into itself, associated with a kernel $\mathcal{K}$ in the sense that
 \begin{equation}\label{sing}
 S(f)(x)=\int_{\mathbb{R}^N}\mathcal{K}(x,y)f(y)dm_k(y),
 \end{equation}
 for all compactly supported function $f$ in $L^2(m_k)$ and
for a.e  $x\in \mathbb{R}^N$ satisfying  $g.x\notin supp(f)$, for all $g\in G$. If $\mathcal{K}$ satisfies \begin{equation}\label{ker}
\int_{\min_{g\in
G}\|g.x-y\|>2\|y-y_0\|}|\mathcal{K}(x,y)-\mathcal{K}(x,y_0)|dm_k(x)\leq
C,\quad y,y_0\in \mathbb{R}^N,
\end{equation} then $S$  extends to a bounded operator from
$L^p(m_k)$ into itself for all  $1<p\leq 2$.
\end{theorem}
\begin{proof}
We first note that $(\mathbb{R}^N,m_k)$ is a space of homogenous type, that is, there is a fixed constant  $C>0$ such that
\begin{equation} \label{dm}
 m_k( B(x,2r))\leq Cm_k(B(x,r)),\quad \forall\;x\in\mathbb{R}^N,\;r>0
\end{equation}
where $B(x,r)$ is  the closed ball  of radius $r$ centered at $x$.  Then we can  adapt to our context  the classical  technic
 which consist to show that $S$ is weak type (1,1) and conclude by Marcinkiewicz interpolation theorem.
\par In fact,
 the   Calder\'{o}n-Zygmund decomposition says that for all
$f\in L^1(m_k)\cap L^2(m_k)$ and $\lambda>0$, there
exist a decomposition of $f$, $f=h+b$ with $b=\sum_j
b_j$ and a sequence of balls $(B(y_j,r_j))_j$ =$(B_j)_j$ such that for some
constant $C$, depending only on the multiplicity function $k$
\begin{itemize}
  \item [(i)]$\norm{h }_\infty\leq C \lambda$;
  \item [(ii)]$supp(b_j)\subset B_j$;
  \item [(iii)]$\displaystyle{\int_{B_j}b_j(x)dm_k(x)=0}$;
  \item [(iv)]$\norm{b_j}_{1,k}\leq
C\;\lambda\, m_k(B_j)$;
\item [(v)]$\displaystyle{\sum_j m_k(B_j)\leq C\;\frac{\norm{f}_{1,k}}{\lambda}}$ .
\end{itemize}
The proof  consists in showing the following inequality hold for $w=h$ and $w=b$ :
\begin{eqnarray}\label{Tg}
\rho_\lambda (S(w))=m_k\Big(\{x\in\mathbb{R}^N;\; |S(w)(x)|>\frac{\lambda}{2}\}\Big)\leq
C\;\frac{\norm{f}_{1,k}}{\lambda}.
\end{eqnarray}
\par By using the $L^2$-boundedness of $S$ we get
\begin{equation}\label{h}
\rho_\lambda (S(h))\leq
\frac{4}{\lambda^2}\int_{\mathbb{R}^N}| S(h)(x)|^2dm_k(x) \leq
\frac{C}{\lambda^2}\int_{\mathbb{R}^N}| h(x)|^2dm_k(x).
\end{equation}
  From $(i)$ and $(v)$,  \begin{equation}\label{h1}\int_{\cup B_j}| h(x)|^{2}dm_k(x)\leq C \lambda^2\mu_k(\cup B_j)\leq C
   \lambda
\norm{f}_{1,k}.\end{equation}
Since  on $(\bigcup_j  B_j)^c$, $f(x)=g(x)$, then
\begin{equation}\label{h2}\int_{(\cup_j B_j)^c}| h(x)|^{2}dm_k(x)\leq C \lambda\norm{f}_{1,k}.\end{equation}
From (\ref{h}),  (\ref{h1}) and (\ref{h2}), the inequality (\ref{Tg}) is satisfied for $h$.
\par Next we turn to the inequality (\ref{Tg}) for the function $b$.
Consider $$B_j^* =B(y_j,2r_j\; );\quad \mbox{ and}  \quad Q_j^*= \bigcup_{g\in G} \;g.B_j^*.$$ Then
\begin{eqnarray*}
\rho_\lambda (S(b))\leq m_k\Big(\bigcup_j Q_j^* \Big)+m_k\set{x\in \Big(\bigcup_jQ_j^*\Big)^c ;|S(b)(x)|>\frac{\lambda}{2}}.
  \end{eqnarray*}
Now by (\ref{dm}) and \textit{(v)}
$$m_k \Big( \bigcup_j Q_j^*\Big)\leq |G| \sum_jm_k(B_j^*)\leq C \sum_jm_k (B_j)\leq
C\;\frac{\norm{f}_{1,k}}{\lambda}.$$
    Furthermore if  $x\notin Q_j^*$, we have $$
\min_{g\in G}\norm{g.x- y_j}>2\norm{y-y_j},\quad  y\in B_j.  $$
Thus, from (\ref{sing}),\textit{(iii)} ,\textit{(ii)}, (\ref{ker}), \textit{(iv)} and \textit{(v)}
\\\\
$\displaystyle{\int_{(\cup Q_j^*)^c}|S(b)(x)|dm_k(x)}$
\begin{eqnarray*}
&\leq&\sum_j
\int_{(Q_j^*)^c}|S(b_j)(x)|dm_k(x)\\
& =&\sum_j\int_{(Q_j^*)^c}\abs{\int_{\mathbb{R}^N} \mathcal{K}(x,y)b_j(y)dm_k(y)}dm_k(x)\\&=&
 \sum_j \int_{(Q_j^*)^c}\abs{\int_{\mathbb{R}^N}b_j(y)\Big(\mathcal{K}(x,y)-\mathcal{K}(x,y_j)\Big)dm_k(y)}dm_k(x)\\
 &\leq&\sum_j\int_{\mathbb{R}^N} |b_j(y)|\int_{(Q_j^*)^c}\abs{\mathcal{K}(x,y)-\mathcal{K}(x,y_j)}dm_k(x)dm_k(y)\\
&\leq& \sum_j\int_{\mathbb{R}^N}|b_j(y)| \int_{\min_{g\in G}\norm{g.x- y_j }>2\norm{y-y_j}}|\mathcal{K}(x,y)-\mathcal{K}(x,y_j)|
dm_k(x)dm_k(y)
\\
&\leq& C \sum_j \norm{b_j}_{1,k}\\
&\leq& C \;\norm{f}_{1,k}.
\end{eqnarray*}
Therefore,
\begin{eqnarray*}
m_k\set{x\in \Big(\bigcup_j Q_j^*\Big)^c ;|S(b)(x)|>\frac{\lambda}{2}}
\leq\frac{2}{\lambda}\int_{(\cup Q_j^*)^c}|S(b)(x)|dm_k(x)\leq C\;\frac{\norm{f}_{1,k}}{\lambda}.
\end{eqnarray*}
This achieves the proof of
(\ref{Tg}) for $b$.
\end{proof}
\par Now, we will give an integral representation for the Riesz transform $\mathcal{R}_j$. For this end, we put for
$x,y\in\mathbb{R}^N$ and $\eta\in co(G.x)$
$$A(x,y,\eta)=\sqrt{\norm{x}^2+\norm{y}^2-2<y,\eta>}=\sqrt{\norm{y-\eta}^2+\norm{x}^2-\norm{\eta}^2}.$$
It is easy to check that
\begin{equation}\label{A}
  \min_{g\in G}\norm{g.x-y}\leq A(x,y,\eta)\leq \max_{g\in G}\norm{g.x-y}.
\end{equation}
The following is clear
 \begin{eqnarray}\label{a1}\quad
\Big|\frac{\partial A^\ell}{\partial y_r}(x,y,\eta)\Big|\leq
 C A^{\ell-1}(x,y,\eta),\; \Big|\frac{\partial^2 A^\ell}{\partial y_r\partial y_s}(x,y,\eta)\Big|\leq
 C A^{\ell-2}(x,y,\eta)
 \end{eqnarray}
 \begin{equation} \label{a2}
\Big|\frac{\partial A^\ell}{\partial y_r}(x,\sigma_\alpha.y,\eta)\Big| \leq
 C A^{\ell-1}(x,\sigma_\alpha.y,\eta),\;  \Big|\frac{\partial^2 A^\ell}{\partial y_r\partial y_s}(x,\sigma_\alpha.y,\eta)\Big|\leq
 C A^{\ell-2}(x,\sigma_\alpha.y,\eta),
 \end{equation}for all $r,s=1,...N$ and $\ell\in\mathbb{R}$.

 Let us set \begin{eqnarray*}\mathcal{K}_j^{(1)}(x,y)&=&\int_{\mathbb{R}^N}
\frac{\eta_j-y_j}{ A^{p_k} (x,y,\eta)}d\mu_x(\eta)\\ \mathcal{K}_j^{(\alpha)}(x,y)&=&\frac{1}{<y,\alpha>}
\int_{\mathbb{R}^N}\Big[\frac{1}{A^{p_k-2}(x,y,\eta)}-\frac{1}{A^{p_k-2}(x,\sigma_\alpha.y,\eta)}\Big]d\mu_x(\eta),\quad
\alpha \in R_+,\\
\mathcal{K}_j(x,y)&=& d_k\Big\{\mathcal{K}_j^{(1)}(x,y)+\sum_{\alpha\in R_+}\frac{k(\alpha)\alpha_j}{p_k-2}
\mathcal{K}_j^{(\alpha)}(x,y)\Big\}.
\end{eqnarray*}
\begin{proposition}\label{l2} If $f\in L^2(m_k)$ with compact support, then for all $x\in\mathbb{R}^N$ such that
$g.x\notin supp(f)$, $g\in G$, we have
$$\mathcal{R}_j(f)(x)=\int_{\mathbb{R}^N}\mathcal{K}_j(x,y)f(y)dm_k(y).$$
\end{proposition}
\begin{proof}
Let $f\in L^2(m_k)$ be a compact supported function  and $x\in \mathbb{R}^N$, such that $g.x\notin supp(f)$ for all $g\in G$.
For $\displaystyle{0<\varepsilon<\min_{g\in G}\min_{y\in supp(f)}|g.x-y|}$ and $n\in\mathbb{N}$, we consider
$ \widetilde{\varphi}_{n,\varepsilon}$  a $C^\infty-$function on $\mathbb{R}$, such that:
\begin{itemize}
\item $\widetilde{\varphi}_{n,\varepsilon}$ is odd .
 \item $\widetilde{\varphi}_{n,\varepsilon}$ is supported in $\set{t\in \mathbb{R}; \;\varepsilon \leq|t|\leq n+1}$.
 \item $\widetilde{\varphi}_{n,\varepsilon}=1$ in $\set{t\in \mathbb{R}; \;\varepsilon+\frac{1}{n}\leq t \leq n}$.
 \item $\abs{\widetilde{\varphi}_{n,\varepsilon}}\leq1$.
 \end{itemize}
Let
$$\widetilde{\phi}_{n,\varepsilon}(t)=\int_{-\infty}^t \frac{\widetilde{\varphi}_{n,\varepsilon}(u)}{|u|^{p_k-1}} \;du\;\quad \mbox{and}
\quad \phi_{n,\varepsilon}(y)=\widetilde{\phi}_{n,\varepsilon}(\norm{ y }),\quad t\in\mathbb{R},\; y\in\mathbb{R}^N.$$ Clearly,
 $\phi_{n,\varepsilon}$  is a $C^\infty$ radial function supported in the ball $B(0,n+1)$ and
   $$\lim_{n\rightarrow +\infty}\widetilde{\varphi}_{n,\varepsilon}(\norm{y})=1, \quad\forall \;y\in\mathbb{R}^N,\;\norm{y}>\varepsilon. $$
 The dominated convergence theorem, (\ref{dtr}) and (\ref{cov}) yield
\begin{eqnarray*}
\int_{\norm{y}>\varepsilon}\tau_x(f)(-y)\frac{y_j}{\norm{y}^{p_k}}dm_k(y)&=&\lim_{n\rightarrow\infty}\int_{\mathbb{R}^N}\tau_x(f)(-y)
\frac{y_j}{\norm{y}^{p_k}}\widetilde{\varphi}_{n,\varepsilon}(\norm{y})dm_k(y)\\
&=&\lim_{n\rightarrow\infty}\int_{\mathbb{R}^N}\tau_x(f)(-y)
T_j(\phi_{n,\varepsilon})(y)dm_k(y)\\
&=&\lim_{n\rightarrow\infty}\int_{\mathbb{R}^N}f(y)
T_j\tau_x(\phi_{n,\varepsilon})(-y)dm_k(y).
\end{eqnarray*}
Now we have
\begin{eqnarray*} T_j\tau_x(\phi_{n,\varepsilon})(-y)&=&
      \int_{\mathbb{R}^N}
\frac{(\eta_j-y_j)\widetilde{\varphi}_{n,\varepsilon}(A(x,y,\eta))}{  A^{p_k}(x,y,\eta) }d\mu_x(\eta)\\ &+&
\sum_{\alpha\in R_+}k(\alpha)\alpha_j\int_{\mathbb{R}^N}\frac{\widetilde{\phi}_{n,\varepsilon}(A(x,\sigma_\alpha.y,\eta))-
\widetilde{\phi}_{n,\varepsilon}(A(x,y,\eta))}
{<y,\alpha>}d\mu_x(\eta),
\end{eqnarray*}
where  from (\ref{A})
 $$\varepsilon<  A(x,y,\eta)\; ; \quad \varepsilon<  A(x,\sigma_\alpha.y,\eta),\quad y\in supp(f),\;\eta\in co(G.x).$$
 Then with the aid of
dominated convergence theorem
$$
\lim_{n\rightarrow\infty} T_j\tau_x(\phi_{n,\varepsilon})(-y) =  \frac{1}{d_k}\mathcal{K}_j(x,y),
 $$
and
$$d_k\int_{\norm{y}\geq\varepsilon}\tau_x(f)(-y)\frac{y_j}{\norm{y}^{p_k}}dm_k(y)=\int_{\mathbb{R}^N}\mathcal{K}_j(x, y)f(y)dm_k(y).$$
Letting  $\varepsilon\rightarrow0$, it follows that
$$\mathcal{R}_j(f)(x)=\int_{\mathbb{R}^N}\mathcal{K}_j(x,y)f(y)dm_k(y),$$
which proves the result.
\end{proof}
Now, we are able to state our main result.
\begin{theorem}\label{th2}
The Riesz transform $\mathcal{R}_j$, $j=1...N$, is a bounded operator from $L^p(m_k)$ into itself, for all $1<p<\infty$.
\end{theorem}
\begin{proof}
 Clearly, from (\ref{rm}) and Plancherel's theorem  $\mathcal{R}_j$ is bounded from $L^2(m_k)$ into itself,
with adjoint operator $\mathcal{R}_j^*=-\mathcal{R}_j$. Thus, via duality it's enough to  consider  the range $1<p\leq2$ and apply
  Theorem \ref{l1}. In view of Proposition \ref{l2} it only remains to show that  $\mathcal{K}_j$ satisfies condition
  (\ref{ker}).\par
Let  $y,y_0\in \mathbb{R}^N$, $y\neq y_0$ and $x\in \mathbb{R}^N$, such that
\begin{equation}\label{yy0}
    \min_{g\in G}\norm{g.x-y}> 2\|y-y_0\|.
\end{equation}
 By mean value theorem,\\\\
 $\displaystyle{\abs{\mathcal{K}_j^{(1)}(x,y)-\mathcal{K}_j^{(1)}(x,y_0)}}$
  \begin{eqnarray*}
 &=&\abs{\sum_{i=0}^{N}(y_i-(y_0)_i)\int_{0}^{1}\frac{\partial\mathcal{K}_j^{(1)}}{\partial y_i }(x,y_t)\;dt}\\&=& \abs{\sum_{i=0}^{N}(y_i-(y_0)_i)\int_{0}^{1}\int_{\mathbb{R}^N}\frac{\delta_{i,j}}{A^{p_k}(x,y_t,\eta)}
+\frac{p_k((y_t)_i-\eta_i)(\eta_j-(y_t)_j)}{ A^{p_k+2}(x,y_t,\eta) }\;d\mu_x(\eta)}\\
&\leq& C \|y-y_0\|\int_0^1\int_{\mathbb{R}^N}\frac{1}{A^{p_k}(x,y_t,\eta)}\;d\mu_x(\eta)dt.
 \end{eqnarray*}
   where $y_t=y_0+t(y-y_0)$ and  $\delta_{i,j}$ is the Kronecker symbol.\\
In view of   (\ref{A}) and (\ref{yy0}), we obtain
$$ \|y-y_0\|< A(x,y_t,\eta),\quad \eta\in co(G.x).$$
 Therefore,\\

 $\displaystyle{\abs{\mathcal{K}_j^{(1)}(x,y)-\mathcal{K}_j^{(1)}(x,y_0)}}$
\begin{eqnarray}\label{k1}
 \nonumber&\leq&C \|y-y_0\|\int_0^1\int_{\mathbb{R}^N}\frac{1}{\Big(\|y-y_0\|^2+A^2(x,y_t,\eta)\Big)^{\frac{p_k}{2}}}\;d\mu_x(\eta)dt.
\\ \nonumber &\leq& C \|y-y_0\|\int_0^1\tau_x(\psi)(y_t)dt
\end{eqnarray}
where $\psi$ is the function defined by
$$\psi(z)=\frac{1}{(\|y-y_0\|^2+\norm{z}^2)^{\frac{p_k}{2}}},\quad z\in\mathbb{R}^N.$$
 Using   Fubini's theorem, (\ref{com}) and   (\ref{tp}), we get\\\\
  $\displaystyle{ \int_{\min_{g\in G}\norm{g.x-y}>2 |y-y_0|}|\mathcal{K}_j^{(1)}(x,y)-\mathcal{K}_j^{(1)}(x,y_0)|dm_k(x)}$
\begin{eqnarray*}
  &\leq& C \|y-y_0\|\int_0^1 \int_{\mathbb{R}^N}\tau_{-y_t}(\psi)(x)dm_k(x)\;dt
\\&\leq& C |y-y_0|\int_{\mathbb{R}^N}\psi(z)dm_k(z)= C\int_{\mathbb{R}^N}\frac{du}{(1+u^2)^{\frac{p_k}{2}}}=C'.
\end{eqnarray*}
This established  the condition (\ref{ker}) for  $\mathcal{K}_j^{(1)}$.
\par To deal with $\mathcal{K}_j^{(\alpha)}$, $\alpha \in R_+$, we
put  for $x,y\in \mathbb{R}^N$ , $\eta\in co(G.x)$ and $t\in[0,1]$
\begin{eqnarray*}
U(x,y,\eta)&=&A^{ 2p_k-4 }(x,y,\eta), \\
 V_\alpha(x,y,\eta)&=&A^{p_k-2}A^{p_k-2}_\alpha(A^{p_k-2}+A^{p_k-2}_{\alpha}),\\
h_{\alpha,t}(y)&=& y+t(\sigma_\alpha.y-y)=y-t<y,\alpha>\alpha,\\
\end{eqnarray*}
 By mean value theorem we have
\begin{eqnarray*}
\mathcal{K}_j^{(\alpha)}(x,y)&=&\int_{\mathbb{R}^N}\frac{1}{<y,\alpha>}\frac{U(x,\sigma_\alpha.y,\eta) -
U(x, y,\eta)}{V_\alpha(x,y,\eta)}d\mu_x(\eta) \\&=&-\int_{\mathbb{R}^N}\int_0^1
\frac{ \partial_\alpha U(x,h_{\alpha,t}(y),\eta)}{V_\alpha(x,y,\eta)}\;dt\;d\mu_x(\eta)
\end{eqnarray*}
and
 \begin{equation}\label{KK1} \mathcal{K}_j^{(\alpha)}(x,y)-\mathcal{K}_j^{(\alpha)}(x,y_0)  =
\int_{\mathbb{R}^N}\int_0^1\int_0^1\partial_{y-y_0}
\Big(\frac{ \partial_\alpha U(x,h_{\alpha,t}(.),\eta)}{V_\alpha(x,.,\eta)}\Big)(y_\theta)\;d\theta \;dt\;d\mu_x(\eta).\end{equation}
Here the derivations are taken with respect to the variable $y$

To simplify, let us denote by $$A=A(x,y_\theta,\eta);\quad A_\alpha=A(x,\sigma_\alpha.y_\theta,\eta)$$
  Then using  (\ref{a1}) and the fact
$  \|\eta-h_{\alpha,t}(y_\theta)\|\leq \max(\;\|\eta-y_\theta\|,\|\eta-\sigma_\alpha(y_\theta)\|\;),$
  we obtain  \begin{eqnarray*}\label{1}
   \Big|\frac{\partial U}{\partial y_r} (x,h_{\alpha,t}(y_\theta),\eta) \Big|&\leq& C \Big(A^{2p_k-5}+
 A^{2p_k-5}_\alpha\Big)\\\label{2}
  \Big|\frac{\partial^2 U}{\partial y_r\partial y_s}(x,h_{\alpha,t}(y_\theta),\eta) \Big|
 &\leq& C  \Big(A^{2p_k-6}+
 A^{2p_k-6}_\alpha\Big),\quad r,s=1,...,N.
 \end{eqnarray*}
  This gives us the following estimates
 \begin{equation}\label{u1}
 \abs{\partial_{\alpha}U(x,h_{\alpha,t}(y_\theta),\eta)}\leq C \Big(A^{2p_k-5}+
 A^{2p_k-5}_\alpha\Big),
 \end{equation}
 \begin{equation}\label{u2}
\abs{\partial_{y-y_0}\Big(\partial_{\alpha}U(x,h_{\alpha,t}(.),\eta)\Big)(y_\theta)}\leq C \|y-y_0\| \Big(A^{2p_k-6}+
 A^{2p_k-6}_\alpha\Big).
 \end{equation}
   By (\ref{a1}) and (\ref{a2}), we also have
   $$\abs{\frac{\partial V_\alpha}{\partial y_r}((x,y_\theta,\eta))}
   \leq C  A^{p_k-3}A^{p_k-3}_\alpha(A^{p_k-2}+A^{p_k-2}_\alpha)(A+A_\alpha).
   $$
The elementary  inequality $\displaystyle{\frac{u+v}{u^\ell+v^\ell}\leq \frac{3}{u^{\ell-1}+v^{\ell-1}}},\quad u, v>0,\;\ell\geq1$, leads to   \begin{eqnarray}\nonumber
   \frac{\abs{\partial_{y-y_0} V_\alpha( x, y_\theta,\eta )}}{V^2_\alpha(x,y_\theta,\eta)} &\leq & C\|y-y_0\|\;\;
   \frac{A_\alpha+A}{A^{p_k-1}A^{p_k-1}_\alpha(A^{p_k-2}+A^{p_k-2}_\alpha)}
   \\\label{v}&\leq&\label{4} C\|y-y_0\| \;\;\frac{1}{A^{p_k-1}A^{p_k-1}_\alpha(A^{p_k-3}+A^{p_k-3}_\alpha)}.
     \end{eqnarray}
  Now  (\ref{u1}), (\ref{u2}) and (\ref{v}) yield\\ \\
  $\displaystyle{\Big|\partial_{y-y_0}
\Big(\frac{ \partial_\alpha U(x,h_{\alpha,t}(.),\eta)}{V_\alpha(x,.,\eta)}\Big)(y_\theta)\Big|}$
   \begin{eqnarray*}
 &\leq& C\|y-y_0\|\,\frac{A^{2p_k-6}+A^{2p_k-6}_\alpha}{A^{p_k-2}A^{p_k-2}_\alpha(A^{p_k-2}+A^{p_k-2}_{\alpha})}\\
&+& C\|y-y_0\|\,\frac{A^{2p_k-5}+A^{2p_k-5}_\alpha}{A^{p_k-1}A^{p_k-1}_\alpha(A^{p_k-3}+A^{p_k-3}_\alpha)}\\ &\leq& C\|y-y_0\|\,
\Big(\frac{1}{A^2A^{p_k-2}_\alpha}+\frac{1}{A^{p_k-2}A^{2}_\alpha}\Big)\\ &+& C\|y-y_0\|\Big(\frac{1}{AA^{p_k-1}_\alpha}+\,\frac{1}{A^{p_k-1}A_\alpha}\Big)\\ &\leq&
C\|y-y_0\|\Big(\frac{1}{A^{p_k}}+\frac{1}{A^{p_k}_\alpha}\Big)
\end{eqnarray*}
where in the last equality we have used the fact that
$\displaystyle{\frac{1}{uv^{\ell-1}}\leq \frac{1}{u^{\ell}}+\frac{1}{v^{\ell}}}$,\\$\quad u, v>0 $ and  $\ell\geq1$.

Thus, in view of (\ref{KK1}),\\

$\displaystyle{\Big|\mathcal{K}_j^{(\alpha)}(x,y)-\mathcal{K}_j^{(\alpha)}(x,y_0)\Big|}$
 $$
  \leq C\|y-y_0\|\int_0^1\int_{\mathbb{R}^N}\Big[\frac{1}{A^{p_k}(x,y_\theta,\eta )}+\frac{1}{A^{p_k}( x,\sigma_\alpha y_\theta,\eta )}\Big]\;d\mu_x(\eta)\;d\theta.
  $$
 Then by same argument as for $\mathcal{K}_j^{(1)}$ we obtain
 $$\int_{\min_{g\in G}|g.x-y |>2 |y-y_0|}|\mathcal{K}_j^{(2)}(x,y)-\mathcal{K}_j^{(2)}(x,y_0)|dm_k(x) \leq C , $$
which  established  the condition (\ref{ker})  for the kernel $\mathcal{K}_j^{(\alpha)}$ and furnishes the proof.
\end{proof}

As applications, we will prove a generalized Riesz and Sobolev
inequalities
\begin{corollary}[Generalized Riesz inequalities]
For all $1<p<\infty$ there exists a constant $C_p$ such that
\begin{eqnarray}\label{ri}
||T_rT_s(f)||_{k,p}\leq C_p||\Delta_k f||_{k,p},\quad \text{for all }\;f\in\mathcal{S}(\mathbb{R}^{N}),
\end{eqnarray}
where $\Delta_k$ is the Dunkl laplacian: $\displaystyle{\Delta_k f=\sum_{r=1}^N T_r^2(f)}$

\end{corollary}
\begin{proof}
From (\ref{dtj}) and (\ref{rm})   one can see that
 $$T_rT_s(f)=\mathcal{R}_r\mathcal{R}_s(-\Delta_k)(f),\; \quad r,\;s=1...N,\; f\in\mathcal{S}(\mathbb{R}^{N}).$$
Then (\ref{ri}) is concluded by Theorem 3.3.
\end{proof}
\begin{corollary}[Generalized Sobolev inequality]For all $1<p\leq q <2\gamma(k)+N$ with $\displaystyle{\frac{1}{q}=\frac{1}{p}-\frac{1}{2\gamma(k)+N}}$, we have
 \begin{eqnarray}\label{SI}
||f||_{q,k}\leq C_{p,q}||\nabla_k f||_{p,k}
 \end{eqnarray}
 for all $f\in\mathcal{S}(\mathbb{R}^{N})$.
 Here $\nabla_k f=(T_1f,...,T_Nf)$ and $\displaystyle{|\nabla_k f|=(\sum_{r=1}^N|T_rf|^2)^{\frac{1}{2}}}$.
 \end{corollary}
\begin{proof}
 For all $f\in\mathcal{S}(\mathbb{R}^{N})$, we write
\begin{eqnarray*}
\mathcal{F}_k(f)(\xi)&=&\frac{1}{\|\xi\|}\sum_{r=1}^N\frac{-i\xi_r}{\|\xi\|}\Big(i\xi_r\mathcal{F}_k(f)(\xi)\Big)
\\&=&\frac{1}{\|\xi\|}\sum_{r=1}^d\frac{-i\xi_r}{\|\xi\|}\Big(\mathcal{F}_k(T_rf)(\xi)\Big).
\end{eqnarray*}
This yields to the following identity
 $$f=I_k^1\Big(\sum_{j=1}^N\mathcal{R}_j(T_jf)\Big),$$ where $$I_k^\beta(f)(x)=(d_{k}^{\beta})^{-1}
\int_{\mathbb{R}^{N}}\frac{\tau_{y}f(x)}
{\|y\|^{2\gamma(k)+N-\beta}}dm_k(y),
$$
here
$$d_{k}^{\beta}=2^{-\gamma(k)-N/2+\beta}
\frac{\Gamma(\frac{\beta}{2})}{\Gamma(\gamma(k)+\frac{N-\beta}{2})}.$$
 Theorem 1.1 of \cite{HS} asserts that $I_k^\beta$ a
bounded operator from $L^p(m_k)$ to $L^q(m_k)$. Then (\ref{SI}) follows from Theorem 3.3.
 \end{proof}

\end{document}